\documentclass[12pt]{article}

\usepackage[utf8]{inputenc} 
\usepackage[T1]{fontenc}
\usepackage[top=2cm, bottom=2cm, left=2cm, right=2cm]{geometry}

\usepackage{amsthm}
\usepackage{amsmath}
\usepackage{amssymb}
\usepackage{mathrsfs}
\usepackage{graphicx}
\usepackage{wrapfig}
\usepackage{makeidx}
\usepackage{dsfont}
\usepackage{stmaryrd}
\usepackage{accents}
\usepackage{bm}
\usepackage{hyperref}
\usepackage{tikz,ulem}
\usetikzlibrary{trees,positioning,decorations.pathreplacing}

\theoremstyle{plain}
\newtheorem{defi}{Definition}
\newtheorem{prop}{Proposition}[section]
\newtheorem{thm}{Theorem}[section]
\newtheorem{lem}{Lemma}[section]

\newtheorem{rem}{Remark}[section]

\newtheorem{theoA}{Theorem}

\newtheorem{propA}{Proposition}

\renewcommand{\l}{\left}
\renewcommand{\r}{\right}

\renewcommand{\o}[1]{\overline{#1}}
\renewcommand{\epsilon}{\varepsilon}

\newcommand{\Id}{\mathrm{Id}}
\newcommand{\Tr}{\mathrm{Tr}}

\makeatletter
\DeclareRobustCommand{\p}[1]{%
  \mathpalette\do@cev{#1}%
}
\newcommand{\do@cev}[2]{%
  \fix@cev{#1}{+}%
  \reflectbox{$\m@th#1\vec{\reflectbox{$\fix@cev{#1}{-}\m@th#1#2\fix@cev{#1}{+}$}}$}%
  \fix@cev{#1}{-}%
}
\newcommand{\fix@cev}[2]{%
  \ifx#1\displaystyle
    \mkern#23mu
  \else
    \ifx#1\textstyle
      \mkern#23mu
    \else
      \ifx#1\scriptstyle
        \mkern#22mu
      \else
        \mkern#22mu
      \fi
    \fi
  \fi
}

\makeatother

\newcommand{\R}{\mathbb{R}}

\newcommand{\calF}{\mathcal{F}}

\newcommand{\G}{\mathcal{G}}

\newcommand{\calE}{\mathcal{E}}

\newcommand{\E}{\mathbb{E}}
\renewcommand{\P}{\mathbb{P}}
\newcommand{\ind}{\mathds{1}}

\title{A multi-dimensional version of Lamperti's relation and the Matsumoto-Yor opposite drift theorem}
\author{Thomas \textsc{Gerard}, Christophe \textsc{Sabot} and Xiaolin \textsc{Zeng}}

\begin{document}

\maketitle

\begin{abstract}
A classic result on the \(1\)-dimensional Brownian motion shows that conditionally on its first hitting time of \(0\), it has the distribution of a \(3\)-dimensional Bessel bridge. By applying a certain time change to this result, Matsumoto and Yor showed in \cite{MY01} a theorem giving a relation between Brownian motions with opposite drifts. The relevant time change is the one appearing in Lamperti's relation.

In \cite{SabZenEDS}, Sabot and Zeng showed that a family of Brownian motions with interacting drifts, conditioned on the vector \(T^0\) of hitting times of \(0\), also had the distribution of independent \(3\)-dimensional Bessel bridges. Moreover, the distribution of \(T^0\) is related to a random potential \(\beta\) that appears in the study of the Vertex Reinforced Jump Process.

The aim of this paper is to show a multivariate version of the Matsumoto-Yor opposite drift theorem, by applying a Lamperti-type time change to the previous family of interacting Brownian motions. Difficulties arise since the time change progresses at different speeds on different coordinates.
\end{abstract}

\section{Introduction}
Let us first recall a well-known result regarding hitting times of the Brownian motion with drift (see \cite{W74} and \cite{V91}).
\begin{propA}
\label{prop-1d-bmhit}
Let \(\theta>0\) and \(\eta\geq 0\) be fixed, and let \(B=(B(t))_{t\ge 0}\) be a standard \(1\)-dimensional Brownian motion. We define the  Brownian motion \(X=(X(t))_{t\ge 0}\) with drift \(\eta\) by
\[X(t)=\theta + B(t) - \eta t,\ t\ge 0.\]
If \(T^0\) is the first hitting time of \(0\) by \(X(t)\), i.e.
\begin{equation}
\label{equation-T0}
T^0=\inf\{t\ge 0,\ B(t)+\theta-\eta t=0\},
\end{equation}
then the distribution of \(T^0\) is given by
\begin{equation}\label{T.dist}
\frac{\theta}{\sqrt{2\pi t^3}}\exp\l( -\frac{1}{2}  \frac{(\theta - \eta t)^2}{t}  \r)\ind_{\{t\geq 0\}}dt.
\end{equation}
Moreover, conditionally on \(T^0\), \(\big( X(t) \big)_{0\leq t\leq T^0}\) has the same distribution as a \(3\)-dimensional Bessel bridge from \(\theta\) to \(0\) on the time interval \([0,T^0]\).
\end{propA}

When \(\eta=0\), \textit{i.e.} \(X\) is a Brownian motion without drift, then \eqref{T.dist} is the density of \(\frac{1}{2\gamma}\), where \(\gamma\) is a Gamma random variable with parameter \((\frac{1}{2},\theta^2)\). This density defines the Inverse Gamma distribution. If \(\eta>0\), \eqref{T.dist} is the density of the Inverse Gaussian distribution with parameter \((\frac{\theta}{\eta},\theta^2)\), which we denote \(\operatorname{IG}(\frac{\theta}{\eta},\theta^2)\).

In \cite{SabZenEDS}, Sabot and Zeng gave a multivariate version of Proposition \ref{prop-1d-bmhit}, which is presented here as Theorem \ref{thm:SDE.X}. The multivariate version concerns a family of Brownian motions indexed by a finite graph with interacting drifts, represented as the solution of a system of stochastic differential equations (SDEs). The hitting times of \(0\) for this family are related to a random potential, which we denote \(\beta\), introduced by Sabot, Tarrès and Zeng in \cite{SabTarZen} and generalized in \cite{Let19}. The distribution of this random potential can be interpreted as a multi-dimensional version of the Inverse Gaussian distribution. It is closely related to the supersymmetric hyperbolic sigma model studied by Disertori, Spencer and Zirnbauer in \cite{DisSpeZir} and \cite{Disertori2010}, and was central in the analysis of the Vertex Reinforced Jump Process in \cite{SabTar}, \cite{SabTarZen} and \cite{SabZen}. See also \cite{Lupu2019a,Bauerschmidt2019b,Bauerschmidt2019,Collevecchio2018,Chang2019,Merkl2019a} for related models in statistical mechanics and random operators.

Another related result is the Matsumoto-Yor opposite drift theorem~\cite{MY01}. This theorem concerns a Brownian motion with negative drift \(-\mu\) which, when conditioned on some exponential functional of its sample path, can be represented as a Brownian motion with opposite drift \(\mu\), with an additional corrective term. A version of this result is stated in Theorem \ref{thm:opp.dr}, in the special case where \(\mu=\frac{1}{2}\). A proof of the theorem relies on applying Lamperti's relation to the classic result on hitting times of the Brownian motion. Lamperti's relation, presented in Proposition \ref{prop:lamperti}, provides a way to write any Bessel process with index \(\mu\) as the exponential of a time-changed Brownian motion with drift \(\mu\).

The aim of this article is to obtain a multi-dimensional version of the opposite drift theorem, by applying an analogue of Lamperti's time change to the family of interacting Brownian motions given in Theorem \ref{thm:SDE.X}. Difficulties arise in applying a time change to the interaction term, since the time change is different on every coordinate of the process. We can overcome this problem in two different ways : either by using the representation given in Theorem \ref{thm:SDE.X}, and applying the time change to each independent 3-dimensional Bessel bridge ; or by using a form of strong Markov property verified by these interacting Brownian motions (c.f. Theorem \ref{thm:Markov.X} or Theorem 2 in \cite{SabZenEDS}).

\section{Statement of the results}
\subsection{Opposite drift theorem and Lamperti's relation}\label{sec:opp.dr.lamp}

Let us present a version of the Matsumoto-Yor opposite drift theorem from \cite{MY01}, in the specific case where the drift \(\mu\) is \(\frac{1}{2}\), and with an added term depending on \(\eta\).

\begin{theoA}\label{thm:opp.dr}[Theorem 2.2 and Proposition 3.1 in \cite{MY01}]
Let \(\theta>0\) and \(\eta\geq 0\) be fixed, and let \(B\) be a standard \(1\)-dimensional Brownian motion. We define the process \(\rho\) as the solution of the following SDE :
\[
\rho(u)= \log(\theta) + B(u) - \frac{1}{2}u -\int_0^u \eta e^{\rho(v)}dv
\]
for \(u\geq 0\).

Let us define \(T(u)=\int_0^u e^{2\rho(v)}dv\). Then :
\begin{itemize}
\item[(i)] We have 
\[
T(u)\xrightarrow[u\to\infty]{a.s.} T^0,
\]
where \(T^0\) is distributed according to 
\[
\frac{\theta}{\sqrt{2\pi t^3}}\exp\l( -\frac{1}{2}  \frac{(\theta - \eta t)^2}{t}  \r)\ind_{\{t\geq 0\}}dt.
\]
\item[(ii)] Conditionally on \(T^0\), there exists a standard \(1\)-dimensional Brownian motion \(\hat{B}\) such that for \(u\geq 0\),
\[
\rho(u)=\log(\theta)+\hat{B}(u)+\frac{1}{2}u + \log\l( \frac{T^0-T(u)}{T^0} \r).
\]
\end{itemize}
\end{theoA}

One proof of Theorem \ref{thm:opp.dr} relies on applying a time change to the introductory result on hitting times of the Brownian motion. The relevant time change is the one that appears in Lamperti's relation, presented below (see e.g. \cite{RY13} p.452) :

\begin{propA}\label{prop:lamperti}[Lamperti's relation]
Let \((\rho(u))_{u\geq 0}\) be a drifted Brownian motion with drift \(\mu\in\R\). For \(u \geq 0\), define
\[
T(u)=\int_0^u \exp(2\rho(u))dv.
\]
Then there exists a Bessel process \((X(t))_{t\geq 0}\) with index \(\mu\), starting from \(1\), such that for \(u\geq 0\),
\[
e^{\rho(u)}(u)=X(T(u)).
\]
\end{propA}

Let us sketch the proof of Theorem \ref{thm:opp.dr} using this time change. We use the same notations as in the introduction : fix \(\theta>0\), \(\eta>0\), and let \(X(t)=\theta+ B(t) +\eta t\) where \(B\) is a standard Brownian motion. Let \(U(t)=\int_0^t \frac{1}{X(s)^2}ds\), and denote \(T=U^{-1}\). Note that \(T\) is the analogue of the time change featured in Lamperti's relation, where \(X\) plays the role of a Bessel process with index \(-\frac{1}{2}\) (with an added drift \(\eta\)).

If \(\rho(u)=\log(X(T(u)))\), then \(e^{\rho(u)}=X(T(u))\), and \(\rho\) has the distribution described in Theorem \ref{thm:opp.dr}, \textit{i.e.} a Brownian motion with drift \(-\frac{1}{2}\), and an extra term when \(\eta>0\). Moreover, when \(u\to\infty\), we have \(T(u)\to T^0\), where \(T^0\) is the first hitting time of \(0\) by \(X\). By Proposition \ref{prop-1d-bmhit}, conditionally on \(T^0\), \(X\) has the distribution of a \(3\)-dimensional Bessel bridge, \textit{i.e.} a Bessel bridge with index \(\frac{1}{2}\), and \(\rho\) has the distribution described in Theorem \ref{thm:opp.dr} (ii).

\begin{rem}
In \cite{MY01}, the opposite drift theorem is stated in a different form, where \(\eta=0\) and the drift \(\mu\) can be different from \(\frac{1}{2}\). Its proof still relies on applying Lamperti's relation, but this time to a result concerning hitting times of Bessel processes with any index \(-\mu\) (see \cite{L18} and \cite{PY81}).
\end{rem}

The aim of this article is now to obtain a multi-dimensional version of Theorem \ref{thm:opp.dr}, by applying the time change from Lamperti's relation to Theorem \ref{thm:SDE.X}, which gives a generalization of the introductory result to a multi-dimensional Brownian motion with interacting drifts.

\subsection{Brownian motions with interacting drifts and the random \(\beta\) potential}

Let \(\G=(V,E)\) be a finite, connected, and non-oriented graph, endowed with conductances \((W_e)_{e\in E}\in(\R_+^*)^E\). For \(i,j\in V\), we denote by \(W_{i,j}=W_{\{i,j\}}\) if \(\{i,j\}\in E\), and \(W_{i,j}=0\) otherwise. Note that it is possible to have \(W_{i,i}>0\). For \(\beta\in\R^V\), we define \(H_\beta=2\beta-W\), where \(W\) is the graph adjacency matrix \(W=(W_{i,j})_{i,j\in V}\), and \(\beta\) denotes here abusively the diagonal matrix with diagonal coefficients \((\beta_i)_{i\in V}\); in particular, \(H_\beta\) is a \(V\times V\) matrix.

\begin{propA}[Theorem 4 in \cite{SabTarZen}, Theorem 2.2 in \cite{Let19}]
\label{prop:nuwtheta}
For all \(\theta\in(\R_+^*)^V\) and \(\eta\in(\R_+)^V\), the measure \(\nu_V^{W,\theta,\eta}\) defined by
\[\nu_V^{W,\theta,\eta}(d\beta)=\ind_{H_\beta>0}\l(\frac{2}{\pi}\r)^{|V|/2}\exp\l(-\frac{1}{2}\langle\theta,H_\beta\theta\rangle-\frac{1}{2}\langle\eta,(H_\beta)^{-1}\eta\rangle+\langle\eta,\theta\rangle\r)\frac{\prod_{i\in V}\theta_i}{\sqrt{|H_\beta|}}d\beta\]
is a probability distribution. Moreover, for all \(i\in V\), the random variable \(\frac{1}{2\beta_i-W_{i,i}}\) has Inverse Gaussian distribution with parameter \((\frac{\theta_i}{\eta_i+\sum_{j\neq i}W_{i,j}\theta_j},\theta_i^2)\).
\end{propA}

For \(t\in(\R_+)^V\), we also denote by \(K_t\) the matrix \(K_t=\Id-tW\), where \(t\) still denotes the diagonal matrix with coefficients \((t_i)_{i\in V}\). Note that if \(t\in(\R_+^*)^V\), we have \(K_t=t H_{\frac{1}{2t}}\), where \(\frac{1}{2t}=\l(\frac{1}{2t_i}\r)_{i\in V}\). Finally, for \(t\in(\R_+)^V\) and \(T\in(\o{\R_+})^V\), we define the vector \(t\wedge T=(t_i\wedge T_i)_{i\in V}\). 

\begin{theoA}\label{thm:SDE.X}[Lemma 1 and Theorem 1 in \cite{SabZenEDS}]
Let \(\theta\in(\R_+^*)^V\) and \(\eta\in(\R_+)^V\) be fixed, and let \((B_i(t))_{i\in V,t\geq 0}\) be a standard \(|V|\)-dimensional Brownian motion.
\begin{itemize}
\item[(i)] The following stochastic differential equation (SDE) has a unique pathwise solution :
\begin{equation}\label{SDE:X}
X_i(t)=\theta_i+\int_0^t\ind_{s<T^0_i}dB_i(s) -\int_0^t\ind_{s<T^0_i}((W\psi)(s)+\eta)_i ds \tag{\(E_V^{W,\theta,\eta}(X)\)}
\end{equation}
for \(i\in V\) and \(t\geq 0\), where for \(i\in V\), \(T^0_i\) is the first hitting time of \(0\) by \(X_i\), and for \(t\geq 0\),
\[
\psi(t)=K_{t\wedge T^0}^{-1}(X(t)+(t\wedge T^0)\eta).
\]

\item[(ii)] If \((X_i)_{i\in V}\) is solution of \eqref{SDE:X}, the vector \(\l(\frac{1}{2 T^0_i}\r)_{i\in V}\) has distribution \(\nu_V^{W,\theta,\eta}\), and conditionally on \((T^0_i)_{i\in V}\), the paths \((X_i(t))_{0\leq t\leq T^0_i}\) are independent \(3\)-dimensional Bessel bridges.
\end{itemize}
\end{theoA}

To obtain an analogue of the opposite drift theorem as in Section \ref{sec:opp.dr.lamp}, we want to apply the time change from Lamperti's relation to solutions \((X_i)_{i\in V}\) of \eqref{SDE:X}. A problem will arise in the interaction term, since the time change will be different on every coordinate of \(X\). To solve this, we will use a form of strong Markov property verified by solutions of \eqref{SDE:X}, which is a consequence of Theorem \ref{thm:SDE.X}(ii). This Markov property will be true with respect to multi-stopping times, defined as follows. 

\begin{defi}\label{defi:multistop}
Let \(X\) be a multi-dimensional càdlag process indexed by \(V\). A random vector \(T=(T_i)_{i\in V} \in \overline{\R_+}^V\) will be called a multi-stopping time with respect to \(X\), if : for all \(t\in\R_+^V\), the event \(\cap_{i\in V}\{T_i\leq t_i\}\) is \(\mathcal{F}^X_t\)-measurable, where 
\[
\mathcal{F}^X_t=\sigma \big( (X_i(s))_{0\leq s\leq t_i}, i\in V \big).
\]
In this case, we denote by \(\mathcal{F}^X_T\) the \(\sigma\)-algebra of events anterior to \(T\), \textit{i.e.}
\[
\mathcal{F}^X_T=\l\{ A\in\mathcal{F}^X_\infty, \forall t\in\R_+^V, A\cap\{T_i\leq t_i\}\in\mathcal{F}^X_{t} \r\}
\]
\end{defi}

Let us now formulate the strong Markov property for solutions of \eqref{SDE:X}.

\begin{theoA}\label{thm:Markov.X}[Theorem 2 (iv) in \cite{SabZenEDS}]
Let \(X\) be a solution of \eqref{SDE:X}, and \(T=(T_i)_{i\in V}\) be a multi-stopping time with respect to \(X\).

Define the shifted process \(Y\) by
\[
Y_i(t)=X_i(T_i+t)
\]
for \(i\in V\) and \(t\geq 0\). Moreover, we denote
\[
\tilde{W}^{(T)}=W\l(K_{T\wedge T^0}\r)^{-1},\tilde{\eta}^{(T)}=\eta+\tilde{W}^{(T)}\big((T\wedge T^0)\eta\big), \text{ and } X(T)=(X_i(T_i))_{i\in V}.
\]

Then on the event \(\cap_{i\in V}\{T_i<\infty\}\), conditionally on \(T\) and \(\mathcal{F}^X_T\), the process \(Y\) has the same distribution as the solution of \(\big(E_V^{\tilde{W}^{(T)},X(T),\tilde{\eta}^{(T)}}(X)\big)\).
\end{theoA}

\subsection{Main results : A multi-dimensional version of the opposite drift theorem}

Let \((X_i)_{i\in V}\) be a solution of \eqref{SDE:X}. As in the usual case of Lamperti's relation, let us introduce the functional that will define the time change. For \(i\in V\) and \(t\geq 0\), we set 
\[
U_i(t)= \int_0^t \frac{\ind_{s<T^0_i}}{X_i(s)^2}ds.
\]
We first have to show that this time change goes to infinity.

\begin{lem}\label{lem:lim.u}
Let \(i\in V\) be fixed. Then
\[
\lim_{t\to T_i^0} U(t)=+\infty,
\]
so that \(U_i:[0,T_i^0[\,\to[0,+\infty[\) is a bijection.
\end{lem}

For all \(i\in V\), we then define \(T_i=(U_i|_{[0,T_i^0[})^{-1}\). Therefore, for all \(u\geq 0\), \(T_i(u)<T^0_i\) and \(\lim_{u\to\infty}T_i(u)=T_i^0\).

We can now show that the time-changed solution \(\big( X_i\circ T_i \big)_{i\in V}\) can be written as
\[
X_i(T_i(u))=e^{\rho_i(u)}
\]
for \(u\geq 0\), where \((\rho_i)_{i\in V}\) is solution of a new sytem of stochastic differential equations :

\begin{thm}\label{thm:SDE.rho}
\begin{itemize}
\item[(i)]
For \(i\in V\) and \(u\geq 0\), let us define \(\rho_i(u)=\log \big( X_i(T_i(u)) \big)\). Then \((\rho,T)\) is solution of the following system of SDEs :
\begin{equation}\label{SDE:rho}
\l\{\begin{aligned}
	\rho_i(v) &= \log(\theta_i) + \tilde{B}_i(v) + \int_0^v \l(-\frac{1}{2}-e^{\rho_i(u)}\l(\tilde{W}^{(u)}(e^{\rho(u)}+T(u)\eta)+\eta\r)_i\r)du,\\
	T_i(v) &= \int_0^v e^{2\rho_i(u)}du,
\end{aligned} \r. \tag{\(E_V^{W,\theta,\eta}(\rho)\)}
\end{equation}
for \(i\in V\) and \(v\geq 0\), where \((\tilde{B}_i)_{i\in V}\) is a \(|V|\)-dimensional standard Brownian motion, \(e^{\rho(u)}\) denotes the vector \((e^{\rho_i(u)})_{i\in V}\), and 
\[
\tilde{W}^{(u)}=W K_{T(u)}^{-1}=W \big( \Id-T(u)W \big)^{-1}.
\]
\item[(ii)] The equation \eqref{SDE:rho} admits a unique pathwise solution \(u\mapsto \big( \rho(u),T(u) \big)\), which is a.s. well defined on all of \(\R_+\).
\end{itemize}
\end{thm}

As a consequence of Theorems \ref{thm:SDE.X}(ii) and \ref{thm:SDE.rho}, we can  relate the solutions of \eqref{SDE:rho} to time-changed Bessel bridges and the distribution \(\nu_V^{W,\theta,\eta}\). This is stated in Theorem \ref{thm:md.opp.dr} below, which is the multi-dimensional version of Theorem \ref{thm:opp.dr}.

\begin{thm}\label{thm:md.opp.dr}
Let \(\tilde{B}\) be a \(|V|\)-dimensional standard Brownian motion, and let \((\rho,T)\) be solution of \eqref{SDE:rho}.
\begin{itemize}
\item[(i)]
For all \(i\in V\), we have
\[
T_i(u)=\int_0^u e^{2\rho_i(v)}dv \xrightarrow[u\to\infty]{a.s.}T_i^0,
\]
where \(\l(\frac{1}{2T_i^0}\r)_{i\in V}\) is distributed according to \(\nu_V^{W,\theta,\eta}\). \item[(ii)]
Conditionally on \(T^0\), there exists a standard \(|V|\)-dimensional Brownian motion \(\hat{B}\) such that for \(i\in V\) and \(u\geq 0\),
\[
\rho_i(u) =\log(\theta_i)+\hat{B}_i(u)+\frac{1}{2}u+\log\l(\frac{T^0_i-T_i(u)}{T^0_i}\r).
\]
In particular, the processes \((\rho_i,T_i)_{i\in V}\) are independent conditionally on \(T^0\).
\end{itemize}
\end{thm}

\subsection{Two open questions}

\subsubsection*{The Matsumoto-Yor property}

The Gamma and Inverse Gaussian distributions, as well as the distribution of the inverse of a Gamma or Inverse Gaussian variable, all fall into the family of the so-called generalized Inverse Gaussian distributions.

A random variable is said to have generalized Inverse Gaussian distribution with parameter \((q, a,b)\) where \(q\in \mathbb{R}\) and \(a,b>0\), and denoted \(\operatorname{GIG}(q, a,b)\) if it has the following density:
\begin{equation}
\label{equation-gig-density}
\left( \frac{a}{b} \right)^{q/2} \frac{1}{2K_q(\sqrt{ab})} t^{q-1} e^{-\frac{1}{2}(at+b/t)}\mathds{1}_{t>0}.
\end{equation}
In particular, we have the following special cases (where zero parameter is understood as in ~\cite{norman1994johnson}):
\[
\operatorname{IG}\l(\frac{\theta}{\eta}, \theta^2\r)=\operatorname{GIG}\l(-\frac{1}{2}, \frac{\eta^2}{2}, \frac{\theta^2}{2}\r),\ \operatorname{Gamma}\l(\frac{1}{2} ,\theta^2\r)=GIG\l(-\frac{1}{2}, 0, \frac{\theta^2}{2}\r)
\]
and
\[
X\sim \operatorname{GIG}\l(-\frac{1}{2},\frac{\eta^2}{2},\frac{\theta^2}{2}\r) \Leftrightarrow 1/X\sim \operatorname{GIG}\l(\frac{1}{2}, \frac{\theta^2}{2}, \frac{\eta^2}{2}\r).
\]
Define the last visit of 0 of our drifted Brownian motion to be $T^1=\sup\{t\ge 0:\ B_t+\theta-\eta t=0\}$.
By a time inversion argument, i.e. setting
\[\widetilde{B}_t=\begin{cases}-tB_{1/t} & t>0 \\ 0 & t=0\end{cases},\]
the Gaussian process \(\widetilde{B}\) is also a Brownian motion and we deduce that \(\frac{1}{T^1}\) is the first visit time to 0 of \(\widetilde{B}_t+\eta-\theta t\), hence \(T^1\) is \(\operatorname{GIG}(\frac{1}{2}, \frac{\theta^2}{2},\frac{\eta^2}{2})\) distributed, moreover, by Strong Markov property of Brownian motion, we deduce that \(T^1-T^0\) is \(\operatorname{Gamma}(\frac{1}{2}, \theta^2)\) distributed, and it is independent of \(T^0\).

More generally, we have the following identity in distribution, which is known as the Matsumoto-Yor property~\cite{Wesolowski2007,matsumoto2003interpretation}:
\begin{prop}
	\label{my-prop}
	Let \(T=(T^0,T^1)\) be a random vector, then there is equivalence between the following statements:
	\begin{itemize}
		\item[(i)] \(\left( \frac{1}{T^0},T^1-T^0  \right)\sim \operatorname{GIG}(\frac{1}{2},\frac{\theta^2}{2},\frac{\eta^2}{2})\otimes \operatorname{Gamma}(\frac{1}{2}, \theta^2)\)
		\item[(ii)] \(\left( \frac{1}{T^0}-\frac{1}{T^1}, T^1 \right)\sim \operatorname{Gamma}(\frac{1}{2}, \eta^2)\otimes \operatorname{GIG}(\frac{1}{2}, \frac{\eta^2}{2},\frac{\theta^2}{2})\).
	\end{itemize}
\end{prop}
It is natural to ask whether such an identity holds true in the case of the distribution in Proposition~\ref{prop:nuwtheta}, and whether one is able to define the process in the time inversion.

\subsubsection*{An opposite-drift theorem for other values of the drift}

The multi-dimensional opposite-drift theorem proved in this paper is limited to the case of drifts $-\frac{1}{2}$ and $\frac{1}{2}$, since it results from Theorem \ref{thm:SDE.X}, which concerns Bessel processes with index $-\frac{1}{2}$ and $\frac{1}{2}$ (\textit{i.e.} Brownian motion and $3$-dimensional Bessel bridges). We could try and obtain a similar result for other values of the drift. This necessitates the use of a random potential analogous to $\beta$, whose marginals would relate to the hitting times of Bessel processes with other indices, \textit{i.e.} generalized Inverse Gaussian variables.

The case of index $\frac{3}{2}$ might be solvable, thanks to recent developments by Bauerschmidt, Crawford, Helmuth and Swan in \cite{BCHS19}, and by Crawford in \cite{Cra19}. These articles concern other sigma models, in particular $\mathbb{H}^{2|4}$, which is related to random spanning forests, and could provide a generalization of the $\beta$ potential corresponding to index $\frac{3}{2}$.

\section{Time change on Bessel bridges}

\subsection{Limit of the time change : Proof of Lemma \ref{lem:lim.u}}

Let \(X=(X_i)_{i\in V}\) be a solution of \eqref{SDE:X}.  According to Theorem \ref{thm:SDE.X}, conditionally on \((T_i^0)_{i\in V}\), the trajectories \((X_i(t))_{0\leq t\leq T_i^0}\) are independent 3-dimensional Bessel bridges. As a result, in order to prove Lemma \ref{lem:lim.u}, it is enough to show the same result for a 3-dimensional Bessel bridge.

Let us then fix \(\theta>0\) and \(T^0>0\), and let \(X\) be a 3-dimensional Bessel bridge from \(\theta\) to \(0\) between \(0\) and \(T^0\). We want to show that
\[U(t)=\int_0^t\frac{ds}{X(s)^2}\xrightarrow[t\to T^0]{a.s.}+\infty.\]
We will do so by applying the time change from Lamperti's relation.

Since \(X\) is a 3-dimensional Bessel bridge, there exists a standard Brownian motion \(B\) such that
\[
dX(t)=dB(t)+\frac{1}{X(t)}dt-\frac{X(t)}{T^0-t}dt,
\]
therefore by Ito's lemma, for \(t<T^0\),
\begin{align*}
d\log(X(t)) &= \frac{dB(t)}{X(t)}+\frac{dt}{X(t)^2}-\frac{dt}{T^0-t} -\frac{1}{2}\frac{dt}{X(t)^2} \\
&= dM(t) +\frac{1}{2}dU(t) +d\log\l(T^0-t\r),
\end{align*}
where \(M(t)=\int_0^t \frac{dB(t)}{X(t)}\) is a martingale, and \(\langle M \rangle_t=U(t)\). Therefore, there exists a standard Brownian motion \(\hat{B}\) such that \(M(t)=\hat{B}(U(t))\). Finally, for \(t\geq 0\), we have
\begin{equation}\label{tc.Bb}
\log(X(t))=\log(\theta)+\hat{B}(U(t))+\frac{1}{2}U(t)+\log\l(\frac{T^0-t}{T^0}\r),
\end{equation}
\textit{i.e.}
\[
\frac{X(t)}{T^0-t}=\frac{\theta}{T^0}e^{\hat{B}(U(t))+\frac{1}{2}U(t)}.
\]

However, since \(X\) is a 3-dimensional Bessel bridge, there also exists a 3-dimensional Bessel process \(Y\) such that for \(t\geq 0\),
\[
X(t)=(T^0-t)Y\l(\frac{t}{T^0(T^0-t)}\r)
\]
(see \cite{RY13} p.467). Therefore, when \(t\to T^0\), we have a.s. \(\frac{X(t)}{T^0-t}\to+\infty\). Since \(u\mapsto\hat{B}(u)+\frac{1}{2}u\) cannot explode in finite time, we have necessarily 
\[
U(t)\xrightarrow[t\to T^0]{a.s.}+\infty.
\]

\subsection{Time change on the conditional process : Proof of Theorem \ref{thm:md.opp.dr}}

Let \((\tilde{B}_i)_{i\in V}\) be a \(|V|\)-dimensional standard Brownian motion. According to Theorem \ref{thm:SDE.rho}, there exists a \(|V|\)-dimensional standard Brownian motion \((B_i)_{i\in V}\) such that, if \((X_i)_{i\in V}\) is the solution of \eqref{SDE:X} with the Brownian motion \(B\), and if \(T_i\) is the inverse function of \(U_i:t\mapsto\int_0^t\frac{ds}{X_i(s)^2}\) for all \(i\in V\), then \((\rho, T)\) is the solution of \eqref{SDE:rho} with the Brownian motion \(\tilde{B}\), where \(\rho_i(u)=\log \big( X_i(T_i(u)) \big)\) for \(u\geq 0\). 

Therefore, according to Lemma \ref{lem:lim.u}, we have a.s. for all \(i\in V\) :
\[
\lim_{u\to +\infty}T_i(u)=T^0_i,
\]
where \(T_i^0\) is the hitting time of \(0\) by \(X_i\). Moreover, we can apply Theorem \ref{thm:SDE.X} (ii) to \(X\) : the vector \(\l(\frac{1}{2T^0_i}\r)_{i\in V}\) is distributed according to \(\nu_V^{W,\theta,\eta}\), and conditionally on \((T^0_i)_{i\in V}\), the trajectories \((X_i(t))_{0\leq t\leq T^0_i}\) are independent 3-dimensional Bessel bridges from \(\theta_i\) to \(0\) respectively. Since \(\rho_i(u)=\log\big(X_i(T_i(u))\big)\) for \(u\geq 0\), conditionally on \((T^0_i)_{i\in V}\), the processes \((\rho_i,T_i)_{i\in V}\) are independent, and their distribution is given by applying the time change from Lamperti's relation to a 3-dimensional Bessel bridge. This time-change was already realized in the proof of Lemma \ref{lem:lim.u} (see \eqref{tc.Bb}), and the result is as follows : conditionally on \((T^0_i)_{i\in V}\), for all \(i\in V\), there exists a standard Brownian motion \(\hat{B}_i\) such that for \(u\geq 0\).
\[
\l\{\begin{aligned}
\rho_i(u)&=\log(\theta_i)+\hat{B}_i(u)+\frac{1}{2}u+\log\l(\frac{T^0_i-T_i(u)}{T^0_i}\r)\\
T_i(u)&=\int_0^u e^{2\rho_i(w)}dw.
\end{aligned}\r.
\]

\section{Multi-dimensional time change : Proof of Theorem \ref{thm:SDE.rho}}

Let us first assume that Theorem \ref{thm:SDE.rho} (i) is proven, and show (ii), \textit{i.e.} that \eqref{SDE:rho} has a.s. a unique pathwise solution defined on all of \(\R_+\). Let \(\tilde{B}\) be a \(|V|\)-dimensional Brownian motion. Thanks to Theorem \ref{thm:SDE.rho} (i), we know that \eqref{SDE:rho} admits a solution that is well defined on \(\R_+\). Let us now show that this solution is necessarily unique.

Let \((\rho^*,T^*)\) be another solution of \eqref{SDE:rho} with the Brownian motion \(\tilde{B}\). Let also \(\mathcal{K}\) be a compact subset of \(\R^V\times \{t\in\R_+^V, K_t>0\}\) containing \((\log(\theta_i),0)_{i\in V}\). Then the function
\[
\l\{\begin{aligned}
\mathcal{K} & \to \R^V\times\R^V \\
(\rho,t) & \mapsto \bigg( -\frac{1}{2} - e^{\rho_i} \Big( W K_t^{-1} (e^\rho+t\eta) +\eta \Big)_i \, , \, e^{2\rho_i} \bigg)_{i\in V}
\end{aligned}\r.
\]
is bounded and Lipschitz. Therefore, up to the stopping time \(U_\mathcal{K}=\inf\{u\geq 0, (\rho(u),T(u))\notin\mathcal{K}\}\), we have \((\rho(u),T(u))=(\rho^*(u),T^*(u))\) from Theorem 2.1, p.375 of \cite{RY13}. Since this is true for all compact subset \(\mathcal{K}\) of \(\R^V\times \{t\in\R_+^V, K_t>0\}\), we have a.s. \((\rho,T)=(\rho^*,T^*)\). This concludes the proof of Theorem \ref{thm:SDE.rho} (ii).

Theorem \ref{thm:SDE.rho} (i) remains to be proven. Let \(B\) be a standard \(|V|\)-dimensional Brownian motion, and let \((X_i)_{i\in V}\) be a solution of \eqref{SDE:X}. For \(i\in V\), recall that \(T_i\) is the inverse function of
\[
U_i:\l\{ \begin{aligned}
\l[0,T^0_i\r[ & \, \to \,  \l[0,+\infty\r[\\
t & \, \mapsto \,  \int_0^t \frac{ds}{X_i(s)^2}
\end{aligned} \r.
\]
and \(\rho_i(u)=\log\big(X_i(T_i(u))\big)\) for \(u\geq 0\). 

In order to show that \((\rho,T)\) is solution of \eqref{SDE:rho}, we want to apply the same time change as in Lamperti's relation. However, in the equation \eqref{SDE:X}, the term \(\psi(t)\) represents an interaction between the coordinates \(X_i(t)\) at the same time \(t\geq 0\), which correspond to different times \(U_i(t)\) when writing \(X_i(t)=e^{\rho_i(U_i(t))}\).

We present here two different ways of overcoming this problem. The first proof relies on identifying the infinitesimal generator of the process \((\rho_i,T_i)_{i\in V}\), using the strong Markov property presented in Theorem \ref{thm:Markov.X}. The second one uses Theorem \ref{thm:SDE.X} (ii), so that we can write \(X\) as a mixture of independent Bessel bridges, to which we can apply the time change separately, and then identify the law of the annealed process using Girsanov's theorem.

\subsection{First proof of (i) : using the strong Markov property of Theorem \ref{thm:Markov.X}}

Firstly, let \(u\geq 0\) be fixed, and \(f:\R^V\times\R^V\to \R\) be a compactly-supported \(\mathcal{C}^2\) function. To identify the infinitesimal generator of \((\rho,T)\), let us compute
\[
\lim_{v\to u^+}\frac{\E[f(\rho(v),T(v))|\mathcal{F}^{(\rho,T)}_u]-f(\rho(u),T(u))}{v-u}.
\]
Note that \((T_i(u))_{i\in V}\) is a multi-stopping time in the sense of Definition \ref{defi:multistop} and that \(\mathcal{F}^{(\rho,T)}_u=\mathcal{F}^X_{T(u)}\). Let us then define
\[
\tilde{W}^{(u)}= W \big( K_{T(u)} \big)^{-1}, \; \tilde{K}^{(u)}_t = \Id-t\tilde{W}^{(u)}, \text{ and } \tilde{\eta}^{(u)}= \eta + \tilde{W}^{(u)}(T(u)\eta).
\]
Thanks to Theorem \ref{thm:Markov.X}, conditionally on \(\mathcal{F}^X_{T(u)}\), the shifted process 
\[
Y=Y^{(u)}:t\mapsto \big( X_i(T_i(u)+t) \big)_{i\in V}
\]
is solution of the following equation :
\[
\l\{\begin{aligned}
dY_i(t) &= \ind_{t\leq\hat{T}^0_i}d\hat{B}_i(t) -\ind_{t\leq\hat{T}^0_i} \l( \tilde{W}^{(u)} (\tilde{K}^{(u)}_t)^{-1} \l(Y(t)+(t\wedge \hat{T}^0)\tilde{\eta}^{(u)}\r) +\tilde{\eta}^{(u)} \r)_i dt,\\
Y_i(0) &= X_i(T_i(u)),
\end{aligned}\r.
\]
for \(i\in V\), \(t\geq 0\), where \(\hat{B}\) is a \(|V|\)-dimensional standard Brownian motion independent from \(\mathcal{F}^X_{T(u)}\), and \(\hat{T}^0_i\) is the first hitting time of \(0\) by \(Y_i\). 

Let us now fix \(v>u\), and define the interrupted process
\[
Z=Z^{(u,v)}:t\mapsto \Big( X_i \big( (T_i(u)+t)\wedge T_i(v) \big)  \Big)_{i\in V}.
\]
For all \(i\in V\) and \(t\geq 0\), we then have \(Z_i(t)=Y_i\l(t\wedge \hat{T}_i(v) \r)\), where \(\hat{T}_i(v)=T_i(v)-T_i(u)\). Therefore, \(Z\) is solution of 
\[
\l\{\begin{aligned}
dZ_i(t) &= \ind_{t\leq \hat{T}_i(v)}d\hat{B}_i(t) -\ind_{t\leq \hat{T}_i(v)}\l( \tilde{W}^{(u)} (\tilde{K}^{(u)}_t)^{-1} \l(Y(t)+(t\wedge \hat{T}^0)\tilde{\eta}^{(u)}\r) +\tilde{\eta}^{(u)} \r)_i dt,\\
Z_i(0) &= X_i(T_i(u)),
\end{aligned}\r.
\]
for \(i\in V\) and \(t\geq 0\). Moreover, since \(\hat{T}_i(v)<\hat{T}_i^0<\infty\) a.s. for all \(i\in V\), there exists a.s. \(\hat{T}^\infty\) large enough so that \(Z_i(t)=Y_i(\hat{T}_i(v))=X_i(T_i(v))\) for all \(i\in V\) and \(t\geq\hat{T}^\infty\).

According to Ito's lemma, for all \(t\geq 0\) we have
\begin{align*}
d\log(Z_i(t)) =&\, \ind_{t\leq \hat{T}_i(v)} \frac{d\hat{B}_i(t)}{Z_i(t)} - \ind_{t\leq \hat{T}_i(v)}\frac{dt}{2 Z_i(t)^2}\\
&  -\ind_{t\leq \hat{T}_i(v)}\l( \tilde{W}^{(u)} (\tilde{K}^{(u)}_t)^{-1} \l(Y(t)+(t\wedge \hat{T}^0)\tilde{\eta}^{(u)}\r) +\tilde{\eta}^{(u)} \r)_i \frac{dt}{Z_i(t)},
\end{align*}
where we can also replace \(t\wedge \hat{T}^0\) with \(t\), since \(\hat{T}_i(v)<\hat{T}^0\). Moreover, we denote by \(\hat{M}_i\) the following martingale for \(i\in V\) : \(\hat{M}_i(t)=\int_0^t \frac{d\hat{B}_i(s)}{Z_i(s)}\) for \(t\geq 0\).

Let us now denote 
\[
\Phi(t)= \big(\log(Z_i(t)), (T_i(u)+t)\wedge T_i(v)\big)_{i\in V}\in\R^V\times\R^V
\]
for \(t\geq 0\). Then, applying Ito's lemma to \(t\mapsto f(\Phi(t))\), we get
\begin{align*}
f(\Phi(t)) &- f(\Phi(0)) = \sum_{i\in V} \int_0^{t\wedge  \hat{T}_i(v)} \frac{\partial f}{\partial \rho_i}(\Phi(s)) d\hat{M}_i(s) + \sum_{i\in V} \int_0^{t\wedge  \hat{T}_i(v)} \frac{1}{2}\frac{\partial^2 f}{\partial\rho_i^2}(\Phi(s))\frac{ds }{Z_i(s)^2} \\
&+ \sum_{i\in V} \int_0^{t\wedge \hat{T}_i(v)} \frac{\partial f}{\partial \rho_i}(\Phi(s)) \l(-\frac{1}{2}- Z_i(s)\l( \tilde{W}^{(u)} (\tilde{K}^{(u)}_s)^{-1} \big(Y(s)+s\tilde{\eta}^{(u)}\big) +\tilde{\eta}^{(u)} \r)_i\r) \frac{ds}{Z_i(s)^2} \\
&+ \sum_{i\in V} \int_0^{t\wedge  \hat{T}_i(v)}\frac{\partial f}{\partial t_i}(\Phi(s))ds.
\end{align*}
Taking \(t\geq\hat{T}^\infty\), we get \(t\wedge \hat{T}_i(v)=\hat{T}_i(v)\) for all \(i\in V\), and 
\[
f(\Phi(t))-f(\Phi(0))=f(\rho(v),T(v))-f(\rho(u),T(u)),
\]
since \(\rho_i(w)=\log\l(X_i(T_i(w))\r)\) for \(w\in\R_+\) and \(i\in V\). For all \(i\in V\), we can now use the following time change in the corresponding integrals above : \(s=T_i(w)-T_i(u)=\hat{T}_i(w)\), \textit{i.e.} \(w=U_i(T_i(u)+s)\). Note that 
\begin{align*}
\text{for } 0\leq s \leq \hat{T}_i(v), \;\; &\frac{d}{ds}U_i(T_i(u)+s) =\frac{1}{X_i(T_i(u)+s)^2}=\frac{1}{Z_i(s)^2},\\
\text{and for } u\leq w \leq v, \;\;  &\frac{d}{dw}\hat{T}_i(w) =X_i(T_i(w))^2=e^{2\rho_i(w)}.
\end{align*}
As a result, we obtain :
\begin{align*}
f\big(\rho(v),&T(v)\big)- f\big(\rho(u),T(u)\big) \\
=&\, \sum_{i\in V} \bigg( \int_u^v \frac{\partial f}{\partial \rho_i}\big(\rho(w),T(w)\big) d\hat{M}_i(\hat{T}_i(w)) + \int_u^v \frac{1}{2}\frac{\partial^2 f}{\partial\rho_i^2}\big(\rho(w),T(w)\big)dw  \\
&+ \int_u^v \frac{\partial f}{\partial \rho_i}\big(\rho(w),T(w)\big) \l(-\frac{1}{2}- e^{\rho_i(w)}\l( \tilde{W}^{(u)} (\tilde{K}^{(u)}_{\hat{T}_i(w)})^{-1} \big(X(T_i(w))+\hat{T}_i(w)\tilde{\eta}^{(u)}\big) +\tilde{\eta}^{(u)} \r)_i\r) dw \\
&+ \int_u^v \frac{\partial f}{\partial t_i}\big(\rho(w),T(w)\big)e^{2\rho_i(w)}dw \bigg).
\end{align*}
Here, the vector \(X(T_i(w))=\big( X_j(T_i(w)) \big)_{j\in V}\) is different from \(e^{\rho(w)}=\big( X_j(T_j(w)) \big)_{j\in V}\). This is why we need to take \(v\to u\) and identify the generator.

Note that since \(\hat{B}\) is independent from  \(\mathcal{F}^{X}_{T(u)}\), we have
\[
\E \l[ \int_u^v \frac{\partial f}{\partial \rho_i}\big(\rho(w),T(w)\big) d\hat{M}_i(\hat{T}_i(w)) \middle| \mathcal{F}^{(\rho,T)}_u\r]=\E\l[ \int_0^{\hat{T}_i(v)} \frac{\partial f}{\partial \rho_i}(\Phi(t)) \frac{d\hat{B}_i(s)}{Z_i(s)}\middle| \mathcal{F}^{X}_{T(u)}\r]=0
\]
for all \(i\in V\), and therefore
\begin{align*}
\E \Big[ f&\big(\rho(v),T(v)\big) -f\big(\rho(u),T(u)\big) \Big| \mathcal{F}^{(\rho,T)}_u \Big] = \E\Bigg[\sum_{i\in V} \Bigg( \int_u^v \frac{1}{2}\frac{\partial^2 f}{\partial\rho_i^2}\big(\rho(w),T(w)\big)dw  \\
&+ \int_u^v \frac{\partial f}{\partial \rho_i}\big(\rho(w),T(w)\big) \l(-\frac{1}{2}- e^{\rho_i(w)}\l( \tilde{W}^{(u)} (\tilde{K}^{(u)}_{\hat{T}_i(w)})^{-1} \big(X(T_i(w))+\hat{T}_i(w)\tilde{\eta}^{(u)}\big) +\tilde{\eta}^{(u)} \r)_i\r) dw \\
&+ \int_u^v \frac{\partial f}{\partial t_i}\big(\rho(w),T(w)\big)e^{2\rho_i(w)}dw \Bigg)\Bigg|\rho(u),T(u)\Bigg].
\end{align*}

By continuity and dominated convergence, we can then conclude that
\begin{align*}
\lim_{v\to u^+} \frac{\E\l[ f\big(\rho(v),T(v)\big) \middle| \mathcal{F}^{(\rho,T)}_u\r]  -f\big(\rho(u),T(u)\big)}{v-u} =& \sum_{i\in V}\Bigg( \frac{1}{2}\frac{\partial^2 f}{\partial\rho_i^2}\big(\rho(u),T(u)\big)\\
+ \frac{\partial f}{\partial \rho_i}\big(\rho(u),T(u)\big) \bigg( -\frac{1}{2}- e^{\rho_i(u)} \big( \tilde{W}^{(u)} e^{\rho(u)} &+\tilde{\eta}^{(u)} \big)_i \bigg) + \frac{\partial f}{\partial t_i}\big(\rho(u),T(u)\big) e^{2\rho_i(u)}\Bigg),
\end{align*}
which is in fact \(\mathcal{L}f(u)\), where \(\mathcal{L}\) is the infinitesimal generator associated with the system of SDEs \eqref{SDE:rho}.

\subsection{Second proof of (i) : using the mixing measure and Girsanov's theorem}

This proof follows the same structure as that of Theorem \ref{thm:SDE.X} : starting from the distribution of the process as a mixture of simpler quenched processes, and computing the integral in order to identify the annealed distribution, using Girsanov's theorem.

Let us denote by \(X=(X_i(t))_{i\in V, t\geq 0}\) the canonical process in \(\mathcal{C}(\R_+,\R^V)\), and by \(\P\) the distribution on \(\mathcal{C}(\R_+,\R^V)\) under which \(X\) is solution of \eqref{SDE:X}. According to Theorem \ref{thm:SDE.X} (ii), the vector \((\beta_i)_{i\in V}=\l(\frac{1}{2T_i^0}\r)\) has distribution \(\nu_V^{W,\theta,\eta}\). Moreover, conditionally on \((T_i^0)_{i\in V}\), the marginal processes \(X_i\) for \(i\in V\) are independent \(3\)-dimensional Bessel bridges from \(\theta_i\) to \(0\) on \([0,T_i^0]\). In other words, we can write
\[
\P[\cdot]=\int \l(\bigotimes_{i\in V} \P_i^{\beta_i} [\cdot] \r) \nu_V^{W,\theta,\eta}(d\beta),
\]
where for \(i\in V\), \(\P_i^{\beta_i}\) is the distribution on \(\mathcal{C}(\R_+,\R)\) under which the canonical process \(X_i\) is a \(3\)-dimensional Bessel bridge from \(\theta_i\) to \(0\) on \([0,T_i^0]\).

We can now apply the time change independently on each marginal \(X_i\) for all \(i\in V\). According to the computations done in the proof of Lemma \ref{lem:lim.u} (see \eqref{tc.Bb}), we know that under \(\P_i^{\beta_i}\), there exists a Brownian motion \(\hat{B}_i\) such that
\[
\rho_i(u) = \log(\theta_i) + \hat{B}_i(u) + \frac{1}{2}u  + \log\l(\frac{T_i^0-T_i(u)}{T_i^0}\r)
\]
for \(u\geq 0\), where \(T_i^0=\frac{1}{2\beta}\) and \(T_i(u)=\int_0^u e^{2\rho_i(v)}dv\), \textit{i.e.}
\[
\rho_i(u) = \log(\theta_i) + \hat{B}_i(u) + \frac{1}{2}u  + \log\l(1-2\beta_i\int_0^u e^{2\rho_i(v)} dv \r).
\]

For each \(i\in V\), let us now define a martingale \(L_i\) by
\[
L_i(u)= \int_0^u \l( -\frac{1}{2} + \frac{2\beta_i e^{2\rho_i(v)}}{1-2\beta_i\int_0^v e^{2\rho_i(s)}ds} \r) d\hat{B}_i(v)
\]
for \(u\geq 0\), so that \(\rho_i(u)= \hat{B}_i(u) - \langle\hat{B}_i,L_i\rangle_u\). We can then introduce a probability distribution \(\hat\P_i\) such that for all \(u\geq 0\),
\[
\E\l[\frac{d\hat\P_i}{d\P_i^{\beta_i}}\middle|\calF^i_u\r]=\calE(L_i)(u),
\]
where \(\calF^i_u=\sigma\l(\hat{B}_i(v),0\leq v\leq u\r)\), and \(\calE(L_i)(u)=e^{L_i(u)-\frac{1}{2}\langle L_i,L_i\rangle_u}\) is the exponential martingale associated with \(L_i\). Then by Girsanov's theorem, \(\rho_i\) is a standard Brownian motion under \(\hat\P_i\). Note that \(\hat\P_i\) does not depend on \(\beta_i\).

From now on, let us write \(\phi_i(u)=1-2\beta_i\int_0^u e^{2\rho_i(v)}dv=\frac{T_i^0-T_i(u)}{T_i^0}\) for \(u\geq 0\) and \(i\in V\). The following lemma gives an expression of \(\calE(L_i)\).

\begin{lem}
For \(i\in V\) and \(u\geq 0\), define
\[
E_i(u)=\exp\l(-\theta_i^2\beta_i +\frac{\beta_i e^{2\rho_i(u)}}{\phi_i(u)} -\frac{1}{2}\rho_i(u)+\frac{1}{8}u\r) \phi_i(u)^{3/2}\sqrt{\theta_i}.
\]
Then \(\calE(L_i)=E_i\).
\end{lem}

\begin{proof}
It suffices to show that \(\frac{dE_i(u)}{E_i(u)}=dL_i(u)\) for all \(u\geq 0\), which will imply that \(E_i=\calE(L_i)\), since \(E_i(0)=1\) almost surely. Note that \(\rho_i(u) = \hat{B}_i(u) + \frac{1}{2}u + \log (\phi_i(u))\), so that
\[
E_i(u)=\exp\l(-\theta_i^2\beta_i +\beta_i\phi_i(u) e^{2\hat{B}_i(u) + u} -\frac{1}{2}\hat{B}_i(u)-\frac{1}{8}u\r) \phi_i(u)\sqrt{\theta_i}.
\]

By Ito's lemma, for \(u\geq 0\) we have
\begin{align*}
dE_i(u) &= \l(2\beta_i\phi_i(u)e^{2\hat{B}_i(u)+u}-\frac{1}{2}\r)E_i(u) d\hat{B}_i(u) \\
& \qquad + \frac{1}{2}\l(\l(2\beta_i\phi_i(u)e^{2\hat{B}_i(u)+u}-\frac{1}{2}\r)^2+4\beta_i\phi_i(u)e^{2\hat{B}_i(u)+u}\r)E_i(u) du \\
& \qquad + \l(\beta_i\l(\phi_i(u)+\phi_i'(u)\r)e^{2\hat{B}_i(u)+u}-\frac{1}{8}+\frac{\phi_i'(u)}{\phi_i(u)}\r)E_i(u) du.
\end{align*}
Since \(\phi_i'(u)=-2\beta_i e^{2\rho_i(u)}=-2\beta_i\phi_i(u)^2 e^{2\hat{B}_i(u)+u}\), we get
\begin{align*}
\frac{dE_i(u)}{E_i(u)} &= \l(-\frac{1}{2}+\frac{2\beta_i e^{2\rho_i(u)}}{\phi_i(u)}\r) d\hat{B}_i(u) +\l( 2\beta_i^2\phi_i(u)^2 e^{4\hat{B}_i(u)+2u} +\frac{1}{8}+\beta_i\phi_i(u)e^{2\hat{B}_i(u)+u} \r. \\
&\l.\qquad + \beta_i\phi_i(u)e^{2\hat{B}_i(u)+u} -2\beta_i^2\phi_i(u)^2 e^{4\hat{B}_i(u)+2u} -\frac{1}{8} -2\beta_i\phi_i(u)e^{2\hat{B}_i(u)+u}\r) du \\
&= dL_i(u).
\end{align*}

\end{proof}

Fix \(u\geq 0\), then for any event \(A_u\in\calF_u^\rho=\sigma\l(\rho(v),0\leq v\leq u\r)\), we have
\begin{align*}
\P[A_u] & = \int \l(\bigotimes_{i\in V} \P_i^{\beta_i}[A_u]\r)\nu_V^{W,\theta,\eta} = \int \int_{A_u} \prod_{i\in V} \l(E_i(u)^{-1}d\hat{\P}_i\r) \nu_V^{W,\theta,\eta}(d\beta) \\
& = \int_{A_u} D(u) d\hat{\P} ,
\end{align*}
where \(\hat{\P}=\bigotimes_{i\in V}\hat{\P}_i\) and for \(u\geq 0\),
\[
D(u)=\int\l(\prod_{i\in V} E_i(u)^{-1}\r)\nu_V^{W,\theta,\eta}(d\beta).
\]
We now have to compute \(D(u)\), and express it as an exponential martingale, in order to apply Girsanov's theorem once again, and identify the distribution of \(\rho\) under \(\P\).

For \(u\geq 0\), we have
\begin{align*}
D(u)&=\int \exp\l( \sum_{i\in V}\l(\theta_i^2\beta_i-\frac{\beta_i e^{2\rho_i(u)}}{\phi_i(u)}+\frac{1}{2}\rho_i(u)-\frac{1}{8}u\r) \r)\frac{1}{\prod_{i\in V}\phi_i(u)^{3/2}\sqrt{\theta_i}} \\
		&\qquad \ind_{H_\beta>0}\l(\frac{2}{\pi}\r)^{|V|/2}\exp\l(-\frac{1}{2}\langle\theta,H_\beta\theta\rangle-\frac{1}{2}\langle\eta,(H_\beta)^{-1}\eta\rangle+\langle\eta,\theta\rangle\r)\frac{\prod_{i\in V}\theta_i d\beta_i}{\sqrt{|H_\beta|}}.
\end{align*}
In order to compute this integral, we will introduce a change of variables, and obtain an integral against the distribution \(\nu_V^{\tilde{W}^{(u)}, \tilde{\theta}^{(u)}, \tilde\eta^{(u)}}\), where \(\tilde{W}^{(u)}\), \(\tilde{\theta}^{(u)}\) and \(\tilde\eta^{(u)}\) are new parameters depending on the trajectory of \(\rho\) up to time \(u\).

Let us introduce the following notations : for \(u\geq 0\),
\[
\l\{ \begin{aligned}
\beta_i^{(u)} &= \frac{1}{2T_i(u)} \text{ for } i\in V\\
H^{(u)} &= 2\beta^{(u)}-W \\
K^{(u)} &= T(u) H^{(u)} = \Id - T(u)W.
\end{aligned}\r. 
\]
From there, we define the new following parameters : 
\[
\l\{ \begin{aligned}
\tilde{W}^{(u)} &= W (K^{(u)})^{-1} =W+W(H^{(u)})^{-1}W \\
\tilde\eta^{(u)} &= \tilde{W}^{(u)}T(u)\eta + \eta\\
\tilde\theta_i^{(u)} &= e^{\rho_i(u)}\text{ for } i\in V
\end{aligned}\r. 
\]
as well as these associated quantities :
\[
\l\{ \begin{aligned}
\tilde{T}_i(u) &= \frac{1}{2\beta_i}-T_i(u) =\frac{\phi_i(u)}{2\beta_i}\text{ for } i\in V \\
\tilde\beta_i^{(u)} &= \frac{1}{2\tilde{T}_i(u)} =\frac{\beta_i}{\phi_i(u)} \text{ for } i\in V\\
\tilde{H}^{(u)} &= 2\tilde\beta^{(u)}-\tilde{W}^{(u)} \\
\tilde{K}^{(u)} &= \tilde{T}(u) \tilde{H}^{(u)} = \Id - \tilde{T}(u)\tilde{W}^{(u)}.
\end{aligned}\r. 
\]

Using these new notations, we can already write
\begin{equation}\label{tech1}
\sum_{i\in V}\frac{\beta_i e^{2\rho_i(u)}}{\phi_i(u)}=\sum_{i\in V}(\tilde\theta_i^{(u)})^2 \tilde\beta_i^{(u)}=\frac{1}{2}\l\langle\tilde\theta^{(u)}, \l(\tilde{H}^{(u)}+\tilde{W}^{(u)}\r) \tilde\theta^{(u)}\r\rangle
\end{equation}
for \(u\geq 0\). Moreover,we will need the following technical lemma in order to express \(D(u)\) as an integral against \(\nu_V^{\tilde{W}^{(u)}, \tilde{\theta}^{(u)}, \tilde\eta^{(u)}}\).

\begin{lem}[Lemma 2 in \cite{SabZenEDS}]\label{lem:tech}
\begin{itemize} For \(u\geq 0\), we have :
\item[(i)] \(K_{1/2\beta}=\tilde{K}^{(u)}K^{(u)}\)
\item[(ii)] \(\tilde\eta^{(u)} = T(u)^{-1}(H^{(u)})^{-1}\eta \)
\item[(iii)] \(\langle\tilde\eta^{(u)},(\tilde{H}^{(u)})^{-1}\tilde\eta^{(u)}\rangle = \langle\eta,H_\beta^{-1}\eta\rangle - \langle\eta,(H^{(u)})^{-1}\eta\rangle.\)
\end{itemize}
\end{lem}

Using Lemma \ref{lem:tech} (i), we get that for \(u\geq 0\),
\[
H_\beta=2\beta K_{1/2\beta}=2\beta\tilde{K}^{(u)}K^{(u)}=2\beta\tilde{T}(u)\tilde{H}^{(u)}K^{(u)},
\]
where \(2\beta_i\tilde{T}_i(u)=1-\frac{\beta_i}{\beta_i^{(u)}}=\phi_i(u)\) for \(i\in V\). Therefore, we have
\begin{equation}\label{tech2}
\prod_{i\in V}\phi_i(u)^{3/2}\sqrt{|H_\beta|}=\prod_{i\in V}\phi_i(u)^2\sqrt{|\tilde{H}^{(u)}|}\sqrt{|K^{(u)}|},
\end{equation}
where 
\[
\frac{d\tilde\beta_i^{(u)}}{d\beta_i}=\frac{1}{\Big(1-\frac{\beta_i}{\beta_i^{(u)}}\Big)^2}=\frac{1}{\phi_i(u)^2}.
\]
Moreover, for all \(u\geq 0\) we have :
\begin{equation}\label{tech3}
\ind_{H_\beta>0}=\ind_{H(u)>0}\ind_{\tilde{H}(u)>0}.
\end{equation}

Combining equations \eqref{tech1}, \eqref{tech2} and \eqref{tech3}, as well as Lemma \ref{lem:tech} (iii), we finally obtain :
\begin{align*}
D(u) &= \l( \int \ind_{\tilde{H}^{(u)}>0}\l(\frac{2}{\pi}\r)^{|V|/2}\exp\l( -\frac{1}{2}\langle\tilde{\theta}^{(u)},\tilde{H}^{(u)}\tilde{\theta}^{(u)}\rangle -\frac{1}{2}\langle\tilde\eta^{(u)},(\tilde{H}^{(u)})^{-1}\tilde\eta^{(u)}\rangle +\langle\tilde\eta^{(u)},\tilde\theta^{(u)}\rangle\r) \r. \\
		&\qquad \l. \frac{\prod_{i\in V}\tilde\theta_i^{(u)}}{\sqrt{|\tilde{H}^{(u)}|}}\prod_{i\in V}\frac{d\tilde\beta_i^{(u)}}{d\beta_i}d\beta_i\r) \ind_{H^{(u)}>0} \exp\l(\frac{1}{2}\langle\theta,W\theta\rangle+\langle\eta,\theta\rangle\r)\prod_{i\in V}\sqrt{\theta_i} \\
		&\qquad \exp\l(-\frac{1}{2}\langle\tilde\theta^{(u)},\tilde{W}^{(u)}\tilde\theta^{(u)}\rangle -\frac{1}{2}\langle\eta,(H^{(u)})^{-1}\eta\rangle -\langle\tilde\eta^{(u)},\tilde\theta^{(u)}\rangle\r)\frac{\prod_{i\in V}\exp\l(\frac{1}{2}\rho_i(u)-\frac{1}{8}u\r)}{\sqrt{|K^{(u)}|}\prod_{i\in V}\tilde\theta_i^{(u)}} \\
	&= \ind_{H^{(u)}>0} \exp\l(-\frac{1}{2}\langle\tilde\theta^{(u)},\tilde{W}^{(u)}\tilde\theta^{(u)}\rangle -\frac{1}{2}\langle\eta,(H^{(u)})^{-1}\eta\rangle -\langle\tilde\eta^{(u)},\tilde\theta^{(u)}\rangle \r) \\
		&\qquad \frac{\prod_{i\in V}\exp\l(-\frac{1}{2}\rho_i(u)-\frac{1}{8}u\r)}{\sqrt{|K^{(u)}|}}\exp\l(\frac{1}{2}\langle\theta,W\theta\rangle+\langle\eta,\theta\rangle\r)\prod_{i\in V}\sqrt{\theta_i},
\end{align*}
since the integral between brackets becomes
\[
\int \nu_V^{\tilde{W}^{(u)},\tilde\theta^{(u)},\tilde\eta^{(u)}}(d\tilde{\beta}^{(u)}) =1.
\]

Let us now show that \(D\) is the exponential martingale associated with a certain \(\calF_u^\rho\)-martingale. By Ito's lemma, for \(u\geq 0\) we have
\begin{align*}
dD(u) &= \sum_{i\in V} \l( -(\tilde{W}^{(u)} e^{\rho(u)})_i e^{\rho_i(u)} - \tilde{\eta}_i^{(u)} e^{\rho_i(u)} -\frac{1}{2}\r) D(u) d\rho_i(u) \\
	&\qquad + \frac{1}{2} \sum_{i\in V} \Biggl( \l( -(\tilde{W}^{(u)} e^{\rho(u)})_i e^{\rho_i(u)} - \tilde{\eta}_i^{(u)} e^{\rho_i(u)} -\frac{1}{2}\r)^2  \\
	&\qquad \qquad +  \l( -(\tilde{W}^{(u)} e^{\rho(u)})_i e^{\rho_i(u)} -\tilde{W}^{(u)}_{i,i} e^{2\rho_i(u)} - \tilde{\eta}_i^{(u)} e^{\rho_i(u)} \r) \Biggr) D(u) du \\
	& \qquad + \l(-\frac{1}{2}\langle e^{\rho(u)},\partial_u (\tilde{W}^{(u)}) e^{\rho(u)}\rangle -\frac{1}{2}\langle \eta,\partial_u (H^{(u)})^{-1}\eta\rangle \r.  \\
	&\qquad\qquad \l. - \langle\partial_u\tilde\eta^{(u)},e^{\rho(u)}\rangle  - \frac{|V|}{8} -\frac{1}{2}\frac{\partial_u|K^{(u)}|}{|K^{(u)}|} \r) D(u) du.
\end{align*}
Since \(H^{(u)}=2\beta^{(u)}-W=1/T(u)-W\), we have
\[
\partial_u (H^{(u)})^{-1}=(H^{(u)})^{-1}T(u)^{-1}\partial_u (T(u)) T(u)^{-1}(H^{(u)})^{-1},
\]
so that, using Lemma \ref{lem:tech} (ii), we get
\begin{align*}
\langle\eta,\partial_u (H^{(u)})^{-1}\eta\rangle &= \langle T(u)^{-1}(H^{(u)})^{-1}\eta,e^{2\rho(u)}T(u)^{-1}(H^{(u)})^{-1}\eta\rangle \\
	&=\langle\tilde\eta^{(u)},e^{2\rho(u)}\tilde\eta^{(u)}\rangle = \sum_{i\in V}(\tilde\eta_i^{(u)})^2 e^{2\rho_i(u)}.
\end{align*}
Moreover, \(\tilde{W}^{(u)}=W+W (H^{(u)})^{-1} W\), therefore 
\begin{align*}
\langle e^{\rho(u)},\partial_u (\tilde{W}^{(u)})e^{\rho(u)}\rangle &= \langle e^{\rho(u)},W (H^{(u)})^{-1}T(u)^{-1} e^{2\rho(u)} T(u)^{-1}(H^{(u)})^{-1} W e^{\rho(u)}\rangle \\
	&= \langle e^{\rho(u)}, \tilde{W}^{(u)}e^{2\rho(u)}\tilde{W}^{(u)} e^{\rho(u)}\rangle = \sum_{i\in V} (\tilde{W}^{(u)}e^{\rho(u)})_i^2 e^{2\rho_i(u)},
\end{align*}
and \(\tilde\eta^{(u)}=\tilde{W}^{(u)}T(u)\eta +\eta\), so
\begin{align*}
\langle\partial_u\tilde\eta^{(u)},e^{\rho(u)}\rangle &= \langle \partial_u(\tilde{W}^{(u)})T(u)\eta +\tilde{W}^{(u)} \partial_u(T(u))\eta,e^{\rho(u)}\rangle \\
&= \langle \tilde{W}^{(u)}e^{2\rho(u)}\tilde{W}^{(u)}T(u)\eta +\tilde{W}^{(u)} e^{2\rho(u)}\eta,e^{\rho(u)}\rangle\\
&= \langle \tilde{W}^{(u)}e^{2\rho(u)}\tilde\eta^{(u)},e^{\rho(u)}\rangle = \sum_{i\in V} \l(\tilde{W}^{(u)}e^{\rho(u)}\r)_i\tilde\eta_i^{(u)}e^{2\rho_i(u)}
\end{align*}
Finally, we have
\begin{align*}
\partial_u |K^{(u)}| &= \Tr\big(|K^{(u)}|(K^{(u)})^{-1}\partial_u K^{(u)}\big)= -|K^{(u)}|\Tr\big(W (K^{(u)})^{-1} e^{2\rho(u)}\big)\\
 &= -|K^{(u)}| \sum_{i\in V}\tilde{W}_{i,i}^{(u)}e^{2\rho_i(u)}.
\end{align*}

Therefore, we get :
\begin{align*}
\frac{dD(u)}{D(u)} &= \sum_{i\in V} \l( - \l(\tilde{W}^{(u)}\l(e^{\rho(u)} +T(u)\eta\r) + \eta \r)_i e^{\rho_i(u)} -\frac{1}{2}\r) d\rho_i(u) \\
 & \;\; + \frac{1}{2}\sum_{i\in V} \l( (\tilde{W}^{(u)}e^{\rho(u)})_i^2 e^{2\rho_i(u)} + (\tilde\eta_i^{(u)})^2 e^{2\rho_i(u)} + \frac{1}{4}  + 2\l(\tilde{W}^{(u)}e^{\rho(u)}\r)_i\tilde\eta_i^{(u)}e^{2\rho_i(u)}\r. \\
 & \qquad + (\tilde{W}^{(u)}e^{\rho(u)})_i e^{\rho_i(u)} + \tilde{\eta}_i^{(u)}e^{\rho_i(u)}  -(\tilde{W}^{(u)}e^{\rho(u)})_i e^{\rho_i(u)} -\tilde{W}_{i,i}^{(u)}e^{2\rho_i(u)} - \tilde{\eta}_i^{(u)}e^{\rho_i(u)} \\
 &\qquad \l. -(\tilde{W}^{(u)}e^{\rho(u)})_i^2 e^{2\rho_i(u)} - (\tilde\eta_i^{(u)})^2 e^{2\rho_i(u)} - 2\l(\tilde{W}^{(u)}e^{\rho(u)}\r)_i\tilde\eta_i^{(u)}e^{2\rho_i(u)} -\frac{1}{4} + \tilde{W}_{i,i}^{(u)}e^{2\rho_i(u)} \r) du \\
 &= \sum_{i\in V} \l( - \l(\tilde{W}^{(u)}\l(e^{\rho(u)} +T(u)\eta\r) + \eta \r)_i e^{\rho_i(u)} -\frac{1}{2}\r) d\rho_i(u)  = d \tilde{L}(u),
\end{align*}
where for \(i\in V\) and \(u\geq 0\),
\[
\tilde{L}_i(u) =\int_0^u \l(-\frac{1}{2} - \l(\tilde{W}(u)\l(e^{\rho(u)} +T(u)\eta\r) + \eta \r)_i e^{\rho_i(u)}\r)d\rho_i(u).
\]
Therefore, \(D\) is the exponential martingale associated with \(\tilde{L}\).

Recall that for \(u\geq 0\) and any event \(A_u\in\calF_u^\rho=\sigma\l(\rho(v),0\leq v\leq u\r)\), we have
\[
\P[A_u] = \int_{A_u} D(u) d\hat{\P},
\]
\textit{i.e.} \(\P\) is such that
\[
\E\l[\frac{d\P}{d\hat\P} \middle| \calF_u^\rho\r]=\calE(\tilde{L})(u)
\]
for all \(u\geq 0\). Moreover, \(\hat{\P}=\bigotimes_{i\in V}\hat{\P}_i\), therefore \(\rho\) is a \(|V|\)-dimensional standard Brownian motion under \(\hat\P\). According to Girsanov's theorem, the process \(\tilde{B}(u)=\rho(u)-\langle\rho,\tilde{L}\rangle_u\) is a standard Brownian motion under \(\P\). In other words, under \(\P\), the process \(\rho\) verifies the following SDE : for all \(i\in V\) and \(u\geq 0\), 
\[
d\rho_i(u)=d\tilde{B}_i(u)-\frac{1}{2}du - \l(\tilde{W}(u)(e^{\rho(u)}+T(u)\eta)+\eta\r)_i e^{\rho_i(u)}du.
\]

\bibliographystyle{plain}
\bibliography{biblio}

\end{document}